\documentclass[leqno,english,11pt]{smfart}

\usepackage{amsmath, amssymb}
\usepackage{amsfonts}
\usepackage{amscd}
\usepackage{enumerate}
\usepackage{pstricks}
\usepackage[dvips]{graphicx}
\usepackage{array,color,colortbl}
\usepackage{pstricks, pst-plot, pst-node}
\usepackage{subfigure}
\usepackage{lscape}
\usepackage{mdwlist}
\usepackage[francais, english]{babel}
\usepackage[all]{xy}
\usepackage{amsthm}
\usepackage[T1]{fontenc}
\usepackage[latin1]{inputenc} 
\usepackage[T1]{fontenc}
\usepackage{multind}
\usepackage{url}
\usepackage{hyperref}

\makeatletter
\gdef\th@break{\normalfont\slshape
  \def\@begintheorem##1##2{\item[%
       \rlap{\vbox{\hbox{\hskip \labelsep\theorem@headerfont ##1\ ##2}%
                   \hbox{\strut}}}]}%
\def\@opargbegintheorem##1##2##3{%
  \item[\rlap{\vbox{\hbox{\hskip \labelsep \theorem@headerfont
                     ##1\ ##2\ ##3}%
                    \hbox{\strut}}}]}}
\makeatother

\theoremstyle{plain}
\newtheorem{theorem}{Theorem}[section]

\newtheorem{proposition}[theorem]{Proposition}
\newtheorem{lemma}[theorem]{Lemma}

\theoremstyle{definition}
\newtheorem{definition}[theorem]{Definition}

\theoremstyle{remark}
\newtheorem{remark}[theorem]{Remark}
\newtheorem{example}[theorem]{Example}
\makeatletter

\DeclareMathOperator{\Aut}{Aut}
\DeclareMathOperator{\chr}{char}

\newcommand{\Ch}[2]{\begin{bmatrix} #1 \\ #2 \end{bmatrix}}
\DeclareMathOperator{\Disc}{Disc}

\DeclareMathOperator{\End}{End}
\DeclareMathOperator{\Gal}{Gal}
\DeclareMathOperator{\GL}{GL}
\DeclareMathOperator{\Hom}{Hom}
\DeclareMathOperator{\im}{Im}
\DeclareMathOperator{\Jac}{Jac}

\DeclareMathOperator{\Sym}{Sym}
\newcommand{\tp}[1]{{^{t}}\!#1}
\DeclareMathOperator{\Tr}{Tr}

\newcommand{\Mm}{\mathsf{M}}

\newcommand{\epsm}{\boldsymbol{\varepsilon}}

\newcommand{\oo}{\mathcal{O}}

\newcommand{\FF}{\mathbb{F}}
\newcommand{\CC}{\mathbb{C}}
\newcommand{\HH}{\mathbb{H}}

\newcommand{\PP}{\mathbb{P}}
\newcommand{\QQ}{\mathbb{Q}}
\newcommand{\RR}{\mathbb{R}}

\newcommand{\ZZ}{\mathbb{Z}}

\newcommand{\bfI}{\mathbf{I}}

\def\ie{\textit{i.e. }}

\usepackage{fancyhdr}
\usepackage{layouts}

\title[Optimal curves]{Optimal curves of genus $1,2$ and $3$}
\alttitle{Courbes optimales de genre $1,2$ et $3$}
	\author{Christophe Ritzenthaler}
\curraddr{Institut de Math{\'e}matiques de Luminy,
         UMR 6206 du CNRS,
         Luminy, Case 907, 13288 Marseille, France.}
\email{ritzenth@iml.univ-mrs.fr}
\thanks{The  author
acknowledges  partially supported by grant MTM2006-11391 from the Spanish MEC and by grant ANR-09-BLAN-0020-01 from the French ANR.}
\date{12 octobre 2010}
\subjclass {Primary 11G20, 11G10 Secondary 14K25, 14H45} 
\keywords{optimal curve, isogeny class, indecomposable polarization, hermitian module, Serre's obstruction, plane quartic, Siegel modular form, Hasse-Weil-Serre bound}

\begin{document}
\maketitle

\begin{abstract}
In this survey, we discuss the problem of the maximum number of points of curves of genus $1,2$ and $3$ over finite fields. 
\end{abstract}

\begin{altabstract}
Nous examinons la question du nombre maximum de points pour les courbes de genre $1,2$ et $3$ sur les corps finis.
\end{altabstract}

\section{Introduction}
The foundations of the theory of equations over finite fields were laid, among others, by mathematicians like Fermat, Euler, Gauss and Jacobi (see \cite{dickson}). 
Subsequently, there was little activity in the field at least until the end of the 19th century and the study of the zeta function of a curve. Initiated by Dedekind, Weber, Artin and Schmidt, this work led to an analogue of the Riemann hypothesis which was proved by Hasse in the case of elliptic curves and then by Weil in general in 1948 (see \cite{weilbook}). The third, modern, period starts in 1980 with the work of Goppa \cite{goppa-art, goppa}. His construction of error-correcting codes with good parameters from curves over finite fields renewed the interest in this theory.\\
With this application in mind, the theory has focused on the maximum number of points of a (projective, geometrically irreducible, non singular) curve of genus $g$ over a finite field $k=\FF_q$, denoted $N_q(g)$. Asymptotic results, \ie values of $N_q(g)/g$ when $g$ goes to infinity and $q$ is fixed, drew attention first,
 but Serre, in his lectures at Harvard \cite{serre-harvard}, gave equal treatment to the `dual' case, \ie  values of $N_q(g)$ when $g$ is fixed and $q$ varies. It quickly appeared that determining $N_q(g)$ was a hard problem and as soon as $g \geq 3$, only sparse values are known (see for instance the web page \url{www.manypoints.org} for the best estimates when $q$ is small).\\  
 In this survey, we are going to describe the main ideas that have been developed  to deal with the cases $1 \leq g \leq 3$. It is interesting to note that for each value of $g$, we will be confronted not only to harder computations but also to a completely new kind of issue. In order to emphasize this progression, we will not consider the actual value of $N_q(g)$ but the following sub-problem. As we shall recall in Section \ref{bounds}, $N_q(g) \leq 1+q+g \lfloor 2 \sqrt{q} \rfloor$ and we can wonder when   $N_q(g)$ reaches this bound.
 If it does, a curve with this number of points is called optimal and we are going to ask for which values of $q$ such curves exist.\\
  When $g \leq 3$, the classical game to prove or disprove the existence of optimal curves is
 \begin{enumerate}
 \item to prove the existence (or not) of an abelian variety $A/k$ with a `good' Weil polynomial (Section \ref{isogenyclass}). This is going to control the number of points on a possible curve $C/k$ such that $\Jac C \simeq A$.
 \item to put a good polarization $a$ on A such that $(A,a)/k$ is geometrically (\ie over $\bar{k}$) the Jacobian of a curve $\bar{C}$ with its canonical polarization (Section \ref{existencepp}).  
  \item to see if $\bar{C}$ admits a model $C/k$ such that $(\Jac C,j) \simeq (A,a)$ (where $j$ is the canonical polarization of $C$). We will see that if $\bar{C}$ is non hyperelliptic, there can be an obstruction to this descent (see Section \ref{optimal}) and for $g=3$, we will propose solutions to address the computation of the obstruction (see Section \ref{g3}).
 \end{enumerate}
  Most ideas we are going to present here are already contained in \cite{serre-harvard} but our proofs for $g=1$ and $2$ are sometimes different from the original ones and take advantage of subsequent simplifications of the theory. \\

\noindent
{\bf Conventions and notation.}
In the following $g \geq 1$ is an integer and  $q=p^n$ with $p$ a prime and $n>0$ an integer. The letter $k$ denotes the finite field $\FF_q$ and $K$ any perfect field. When we speak about a \emph{genus $g$ curve} we mean that the curve is projective, geometrically irreducible and non-singular. 
If $A$ and $B$ are varieties over a field $K$, when we speak of a morphism from $A$ to $B$ we always mean  a morphism \emph{defined over $K$}. So, for instance $\End(A)$ is the ring of endomorphisms defined over $K$, $A \sim B$ means $A$ isogenous to $B$ over $K$, etc. If $(A,a)$ and $(B,b)$ are polarized abelian varieties, by an isomorphism between them, we always mean `as polarized abelian varieties'. \\ 

\noindent
{\bf Acknowledgements.}
 I would like to thank Christian Maire for suggesting me to write this survey. This is part of my `habilitation' thesis \cite{HDRritz} which was defended during the workshop Theory of Numbers and Applications which was organized by Karim Belabas and Christian Maire in Luminy in December 2009. I am really grateful to Detlev Hoffmann for the references of Remark \ref{cashigher} and to the Number Theory List community and particularly to Samir Siksek for helping me with Remark \ref{m2}. 

\section{Control of the isogeny class} \label{isogenyclass}

\subsection{Bounds} \label{bounds}
Let $C/k$ be a genus $g$ curve. We recall that its \emph{Weil polynomial} $\chi_C$ is the Weil polynomial of $\Jac C/k$, \ie the characteristic polynomial of the action of the $k$-Frobenius endomorphism on an $\ell$-adic Tate module for any prime $\ell \ne p$. It is well known that it can be written
 $$\chi_C=\prod_{i=1}^g (X^2+x_i X+q) \in \ZZ[X]$$ with $x_i \in \RR$ and $|x_i| \leq 2\sqrt{q}$. Since 
$$\# C(k)=q+1+\sum_{i=1}^g x_i,$$
it is clear that $\# C(k) \leq 1+q+ \lfloor 2g \sqrt{q} \rfloor$, which is known as \emph{Hasse-Weil bound} \cite{weilbook} and so $N_q(g)$ is less than this bound too. It is possible to improve this bound as the following lemma shows.
\begin{lemma}[{Hasse-Weil-Serre bound \cite{serre-point1}}]
Let $m=\lfloor  2 \sqrt{q}\rfloor$. Then $$N_q(g) \leq 1+q+gm.$$
\end{lemma}
\begin{proof}
It is enough to use the arithmetic-geometric mean inequality:
$$\frac{1}{g} \sum_{i=1}^g (m+1-x_i) \geq \left(\prod_{i=1}^g (m+1-x_i)\right)^{1/g} \geq 1,$$
the last inequality coming from the fact that the product is a non-zero integer.
\end{proof}
This motivates us to give the following definition.
\begin{definition}
We say that a genus $g$ curve $C/\FF_q$  is \emph{optimal} if $$\# C(\FF_q)=q+1+gm.$$ In that case $N_q(g)=q+1+gm$.
\end{definition}

Note that the previous definition is not universally accepted. Some authors call \emph{maximal} (or $\FF_q$-maximal) what we call optimal by reference to the historical cases with $n$ even and $N_q(g)=q+1+gm$. We prefer to keep the word maximal for curves which numbers of points is equal to $N_q(g)$ and our terminology is coherent with the historical one as well.

\begin{remark}
If $g \geq (q-\sqrt{q})/2$, the bound can be improved, thanks to the \emph{explicit methods} of Oesterl\'e (and is known as \emph{Oesterl\'e bound} \cite{serre-point1}). It uses the fact that the number of places of each degree on the curve is non negative. As we will mainly deal  with small values of $g$ compared to $q$, the Hasse-Weil-Serre bound will be our reference.
\end{remark}

\subsection{Existence of the isogeny class}
Equality in the arithmetic-geometric mean inequality is equivalent to the fact that all terms in the sum are equal and so $x_i=m$ for all $1\leq i\leq g$. Hence, if an optimal curve $C$ exists, its Weil polynomial has the particular simple expression
$$\chi_C=(X^2+mX+q)^g.$$
Honda-Tate theory as explained in \cite{tate}, \cite{honda}, \cite{waterhouse}, \cite{wami} or \cite{tate2}  shows that if $p \nmid m$ (resp. $n$ is even) then $\Jac C$ is isogenous to $E^g$ where $E$ is an ordinary (resp. supersingular) elliptic curve with trace $-m$. However, if $p | m$ and $n$ is odd, this might not be true (see the proof of Proposition \ref{prop-deu} below) and there is for instance a simple abelian variety of dimension $9$ over $\FF_{5^9}$ with such Weil polynomial. If we restrict to $g \leq 3$, it can be proved that this never happens (see for instance the proof of Corollary 4.2 of \cite{NR09}).

\begin{lemma} \label{toE3}
If $C/\FF_q$ is an optimal curve of genus $g \leq 3$ then $\Jac C$ is isogenous to $E^g$ where $E$ is an elliptic curve of trace $-m$.
\end{lemma}

The first necessary condition is then to see whether such an elliptic curve exists or not.
\begin{proposition}[{Deuring \cite{deuring}}] \label{prop-deu}
There does not exist an elliptic curve with trace $-m$ if and only if $n \geq 3$, $n$ is odd and $p | m$.
\end{proposition}
\begin{proof}
Let $F=X^2+mX+q$. Since $m<2 \sqrt{q}$ if and only if $q$ is not a square, $F$ is irreducible over $\QQ$ when $n$ is odd and $F=(X+\sqrt{q})^2$ when $n$ is even.\\
If $n$ is odd, by \cite[p.527]{waterhouse}, the minimal $e$ for which $\chi=F^e$ is the Weil polynomial of an abelian variety of dimension $e$ over $k$ is the least common denominator of $v_p(F_{\nu}(0))/n$ where $F_{\nu}$ denotes the factors of $F$ in $\QQ_p[t]$ and $v_p$ the $p$-adic valuation of $\QQ_p$. Hence $F$ is the Weil polynomial of an elliptic curve if and only if $n | v_p(F_{\nu}(0))$ for all factors. This is of course satisfied if $n=1$. Looking at the Newton polygon of $F$, we see that if $p \nmid m$ then
 $v_p(F_{\nu}(0))=n$ or $0$, so $e=1$. With the same technique, if $n>1$ odd and $p | m$, then
 $v_p(F_{\nu}(0))<n$ and so $e>1$.\\
If $n$ is even, we apply the previous arguments to $F=X+\sqrt{q}$. Since $v_p(\sqrt{q})/n=1/2$,  $e=2$ so $F^2=X^2+mX+q$ is the Weil polynomial of an elliptic curve.
\end{proof}
Actually, the only values of $q=p$ for which $p | m$ are $q=2$ or $q=3$.
\begin{remark} \label{isoclasses}
For any value of $-m \leq t \leq m$, one knows if an elliptic curve with trace $t$ exists (see \cite[Th.4.1]{waterhouse}). Also, the possible Weil polynomials of the isogeny classes of abelian surfaces (resp. threefolds) can be found in \cite[Lem.2.1,Th.2.9]{maisner} (resp. \cite{xing,haloui}).
\end{remark}  

\section{Existence of an indecomposable principal polarization} \label{existencepp}
The Jacobian of a genus $g$ curve $C/K$ is naturally equipped with a \emph{principal polarization} $j$ induced by the intersection pairing on the curve $C$. Since the theta divisor $\Sym^{g-1} C \hookrightarrow \Jac C$ associated to $j$ is geometrically irreducible, $(\Jac C,j)$ is geometrically \emph{indecomposable}, \ie there does not exists an abelian subvariety $B \subset \Jac C$ defined over $\bar{K}$ such that $j$ induces on $B$ a principal polarization. Conversely, starting with $A=E^g$ where $E$ is an elliptic curve, it is clear that $A$ always admits a principal polarization $a_0$ given by the product of the principal polarizations on each factor. As  $a_0$ is   decomposable, $(A,a_0)$ is not (even geometrically) a Jacobian. Hence `good' principal polarizations on $A$ (or on abelian varieties in the isogeny class of $A$) have to be more subtle. Luckily, equivalences of category have been developed to translate the existence of an indecomposable polarization into the existence of purely algebraic objects. As far as I know several points of view co-exist and it is not clear to see how to go from one to the other. I shall use Serre's one and mention others in remark.

\begin{remark}
Howe  \cite{howe95},\cite{howe96}  has developed a powerful machinery to prove the existence of a principally polarized abelian variety in the isogeny class of an abelian variety $A/k$. But only when  $A$ is  simple, it is easy to see that the polarization is indecomposable (see  \cite{rybakov} for the case $E \times B$ where $E$ is an elliptic curve and $B$ a geometrically simple abelian surface).
\end{remark}

\subsection{The equivalences} \label{equivalence}
Let us start with $E$ ordinary.
Let $E/k$ be an ordinary elliptic curve with trace $t$. If $\pi$ denotes the $\FF_q$-Frobenius endomorphism of the curve $E$, then the ring $R:=\ZZ[X]/(X^2-tX+q)$ is isomorphic to  $\ZZ[\pi] \subset \End(E)$. Serre \cite[Se.50-53]{serre-harvard}, \cite[Appendix]{lauterg3} defines an equivalence of category $T$ between the category of abelian varieties which are isogenous to a power of $E$ and $R$-modules of finite type without torsion. The functor $T$ maps an object $A$ to the $R$-module $L=\Hom(E,A)$. Obviously, the rank of $L$ is equal to the dimension of $A$. This functor also behaves nicely with respect to duality: if we denote $\hat{L}$ the ring of anti-linear homomorphism $f : L \to R$ (\ie $f(rx)=\bar{r} f(x)$ for all $r \in R$ and $x \in L$) then $T(\hat{A})=\hat{L}$.  Thus a morphism $a : A \to \hat{A}$ defines a morphism $h : L \to \hat{L}$ and hence an hermitian form $H : L \times L \to R$. Serre proves that $a$ is a polarization if and only if $H$ is positive definite, that $a$ is principal if $(L,H)$ is \emph{unimodular} (\ie $h(L)=\hat{L}$) and moreover
 (geometrically) indecomposable if and only if  $(L,H)$ is \emph{indecomposable}, \ie cannot be written as a sum of orthogonal sub-modules. The couple $(L,H)$ is called a \emph{ hermitian module}.

\begin{remark} \label{othereq}
This equivalence is inspired by the classical theory over $\CC$, which is not surprising since ordinary abelian varieties can be lifted canonically and this is used in \cite{deligne}. When $A=E^g$ and $\End(E) \simeq R$, a more explicit point of view can be considered looking at the hermitian matrix $M:=a_0^{-1} a \in \End(A)=\Mm_g(\End(E)) \simeq \Mm_g(R)$ (see \cite{Rit-serre}, \cite{langepp}).  For $g=2$, Kani's construction \cite{kani} also gives necessary and sufficient conditions for `gluing' two elliptic curves along their $n$-torsion for $n>1$. Both points of view are related by the Cholewsky decomposition of $M$. 
\end{remark}

A classification of rank $2$ and $3$ hermitian modules was achieved in \cite[Th.8.1,8.2]{hoffmann} (see also \cite{schiemann} for further computations) and translates into the following result.

\begin{proposition} \label{hoff}
Let $E$ be an ordinary elliptic curve with trace $t$. There is no abelian surface (resp. threefold) with a geometrically indecomposable principal polarization in the class of $E^2$ (resp. $E^3$) if and only if $t^2-4q \in \{-3,-4,-7\}$ (resp. $t^2-4q \in \{-3,-4,-8,-11\}$).
\end{proposition}

\begin{remark}
For $g=2$, the result can be traced back to \cite[p.14]{hayashida}, where the authors prove the existence of genus $2$ curves which Jacobian is \emph{isomorphic} to $E^2$ by constructing free indecomposable hermitian modules (in \cite{hayashida2}, the precise number of isomorphism classes of such curves is computed).
 For $g=2$ or $3$, it could also be deduced from the mass formulae (\ie number of weighted classes by the order of their automorphism group)  of \cite{hashi-quad1,hashi-quad2} (although, according to Hoffmann (\textit{loc. cit.} p.400) there is a minor mistake in these computations). \end{remark}

\begin{remark} \label{cashigher}
For $g >3$, there have been several partial answers on the existence of  indecomposable unimodular positive definite hermitian modules of rank $g$ over the ring of integers of an imaginary quadratic field $\QQ(\sqrt{-d})$. It seems that in \cite{zhu} and \cite{wang} a complete answer is given: there always exists one, except when $d=1$ and $g=5$ or $d=3$ and $g=4,5,7$. One should be careful since, according to the Mathscinet review  of \cite{zhu} by Hoffmann, the proofs contain several mistakes. Also, I do not know if the case of non maximal orders has been considered.
\end{remark}

Assume now that $E$ is supersingular. More precisely, let $E/\FF_p$ be an elliptic curve with trace $0$, so that $E$ is supersingular, all the geometric automorphisms of $E$ are defined over $\FF_{p^2}$ and $\Tr(E/\FF_{p^{2}})=-2 p=-m$. One says that an abelian variety $A$ (resp. a curve $C$) is \emph{superspecial} if $A$ (resp. $\Jac C$) is geometrically isomorphic to a product of supersingular elliptic curves. A result of Deligne (see \cite[Th.3.5]{shiodass}) shows that when $g>1$, a superspecial abelian variety of dimension $g$ is geometrically isomorphic to $E^g$ (whereas for $g=1$ there are non-isomorphic supersingular elliptic curves as soon as $p>7$). However, the description of the isogeny class is made more complicated than in  the ordinary case by  the existence of `continuous' families of isogenies. For instance, already when $g=2$ (see \cite{oort75}), a supersingular abelian surface is either geometrically isomorphic to $E^2$ (and so superspecial) or of the form $E^2/\alpha_p$ where $\alpha_p$ is the unique local-local group scheme over $\FF_p$, the injection of $\alpha_p$ in $E^2$ being parametrized by $\PP^1(\bar{\FF}_p)\setminus \PP^1(\FF_{p^2})$. In the latter, it can be shown that $A$ is not superspecial and the description of the polarizations on this object is more evolved. For this reason, we will concentrate only on existence results and limit ourselves to the superspecial case.

\begin{remark}
Note that it is still possible to obtain a complete  description for $g=2$  in the non-superspecial case like in \cite{IKO} or \cite[Part.2]{HNR09} where the mass formula  of \cite{ibuauto}  were used.
\end{remark}

As in Remark \ref{othereq}, we describe the polarizations on $A=E^g$ by matrices $M:=a_0^{-1} a$ in $\End(A)=\Mm_g(\End(E))$. Now, $\End(E)$ is a quaternion algebra, so we need results on the number $n_g$ of \emph{positive definite quaternion hermitian forms}. Then, to obtain the number of (geometrically) indecomposable polarizations on $E^g$, 
the idea is to subtract to $n_g$ the number of polarizations coming from combinations of lower dimensional abelian varieties. In this way, one gets

\begin{proposition}[{\cite[Prop.7.5]{ekedahl}}] \label{ibu-prop}
There is no geometrically indecomposable principal polarization on $E^g$ if and only if $g=2$ and $p=2$ or $3$, or $g=3$ and $p=2$.
\end{proposition}

\begin{remark}
More precisely, Ekedahl gives in \cite[Prop.7.2]{ekedahl} the mass of indecomposable principal polarizations on $E^g$. However, Brock \cite[Th3.10.c]{brock} corrects a mistake in the case $g=3$. He also completes and recovers several results obtained in  \cite[I]{hashimoto},  \cite{KO} for $g=2$  and in \cite{hashimotog3}, \cite{oort89} for $g=3$. For instance, in \cite[Th.3.14,Th.3.15]{brock}, he gives the number of genus $2$ and genus $3$  superspecial curves for each possible group of automorphisms.
\end{remark}

\subsection{Application}
We can now answer the question of the existence of a good polarization when $g \leq 3$.
\begin{theorem} \label{theo-pp}
Let $E$ be an elliptic curve with trace $-m$. There is no abelian surface (resp. threefold) with a geometrically indecomposable principal polarization in the isogeny class of $E^2$ (resp. $E^3$) if and only if $q=4$ or $9$ or $m^2-4q \in \{-3,-4,-7\}$ (resp. $q=4$ or $16$ or $m^2-4q \in \{-3,-4,-8,-11\}$).
\end{theorem}
\begin{proof}
When  $p \nmid m$, $E$ is ordinary so we can use Proposition \ref{hoff}.\\
When $n$ is odd and $p | m$, $E$ exists if and only if $q=2$ or $3$, which leaves these two cases to be treated apart (for instance by extensive computer research of curves using Theorem \ref{optimal} or by Remark \ref{m2}). \\
When $n$ is even then $p | m$. We distinguish several cases.
\begin{itemize}
\item  When $p>3$ and $g=2$ (resp. $p>2$ and g$=3$), Proposition \ref{ibu-prop} shows that there is always an indecomposable principal polarization on $E^2$ (resp. $E^3$). Note than when $4 | n$, the present $E$ is the quadratic twist of the elliptic curve in Proposition \ref{ibu-prop}.
\item  When $p=2$ and $g=2$, explicit constructions as in \cite{MN07} or  a `gluing' argument as in \cite[Se.32]{serre-harvard}, \cite[Prop.30]{shabat}  shows that one can get a curve $C/\FF_{2^n}$ such that $\Jac C$ is isogenous to $E^2$ as soon as $n>2$.
\item   When  $p=2$, $g=3$ and $n>4$, An explicit non hyperelliptic curve $C/\FF_{2^n}$ such that $\Jac C \sim E^3$ can be constructed. (see  \cite[Lem.2.3.8]{HDRritz}). Note that in \cite{NR08}, the more general question of the existence of a Jacobian in the isogeny class of a supersingular abelian threefold in characteristic $2$ is addressed.
\item   Finally when $p=3$ and $g=2$, one can find an explicit construction in \cite{kuhn} (see also \cite[Cor.37]{shabat} where a mistake is corrected) as soon as $n>2$. This work was also generalized to all supersingular abelian surfaces in characteristic $3$ in \cite{hocar3}. 
\end{itemize}
\end{proof}

\begin{remark} \label{m2}
The cases $m^2-4q \in \{-3,-4\}$ can also be excluded thanks to a proof due to Beauville \cite[Th.16]{shabat}, \cite[Se.13]{serre-harvard} without any hypothesis on the $p$-rank of $E$ (and then $q=2,3$ are covered).\\ 
As conjecturally, there is infinitely many $p$ in the forms $p=x^2+1$ and $p=x^2+x+1$ the equations $m^2-4p^n \in \{-3,-4\}$ have infinitely many solutions with $n=1$. 
For $n >1$ odd, one knows that the set of solutions is finite. For instance, in \cite{serre-point}, we find that
 there is only one solution to the equation $q=x^2+x+1$ namely $q=7^3$ and none to $q=x^2+1$.\\
Similarly, the case of discriminant $-7$ corresponds to the equation $q=x^2+x+2$ with unique solutions $q \in \{2^3,2^5,2^{13}\}$ when $n>1$ is odd.\\
The case of  discriminant $-8$ corresponds to $q=x^2+2$, which, when $n>1$ is odd, has $q=3^3$ for unique solution. This is proved using the same arguments as the last case below.\\
Finally, the case of discriminant $-11$ corresponds to $q=x^2+x+3$. When $n>1$ is odd, $q=3^5$ is the unique solution thanks to the following argument due to Samir Siksek. The equation can be rewritten $(2x+1)^2+11=4 p^n$ and  factors in $K=\QQ(\sqrt{-11})$ as $$\left(\frac{2x+1+\sqrt{-11}}{2}\right)
\left(\frac{2x+1-\sqrt{-11}}{2}\right)=p^n.$$
Since $n$ is odd and $\oo_K$ is a principal domain, there exists $\alpha=(a+b \sqrt{-11})/2  \in \oo_K$ such that $\alpha^n=(2x+1+\sqrt{-11})/2$ and $\alpha \beta=p$ with $\beta=\bar{\alpha}$. Now, note that $\alpha^n-\beta^n=\sqrt{-11}$ and since $(\alpha^n-\beta^n)/(\alpha-\beta) = 1/b \in \oo_K$, we see that $b=\pm 1$. Hence, if we fix $n$, we can find the finite set of integer solutions of this polynomial equation in $a$. However, to solve it for all $n$ we have to invoke the much deeper theorem from \cite{bilu} which tells us that if there is a solution then $n<4$, $n=5$ or $n=12$. Indeed, with the terminology and notation of \textit{loc. cit.}, one sees that $u_n=(\alpha^n-\beta^n)/(\alpha-\beta)=\pm 1$ is a Lucas number without primitive divisor, so the Lucas pair $(\alpha,\beta)$ is $n$-defective.
\end{remark}

\section{Optimal curves} \label{optimal}
As we have seen in Section \ref{existencepp}, the strategy we have applied so far works in any dimension. If we now have to restrict ourselves to the dimensions less than or equal to $3$ is because, in these cases, the condition `has an indecomposable principal polarization' is geometrically sufficient to be the Jacobian of a curve. This  is not true  when the dimension is bigger, as it is proved simply by noting that the dimension of the moduli space of curves of genus $g$, $3g-3$, is less than the dimension of the moduli space of principally polarized abelian varieties of dimension $g$, $g(g+1)/2$. However,
\begin{proposition}[\cite{oort-ueno}] \label{ou}
For $g \leq 3$, any geometrically indecomposable principally polarized abelian variety $(A,a)/K$ is the Jacobian of a curve $\bar{C}$ over $\bar{K}$.
\end{proposition}


So, given $(A,a)/K$ as in Proposition \ref{ou}, the question boils down to know whether one can descend the curve $\bar{C}$ to a curve $C$ over $K$ such that $(\Jac C,j) \simeq (A,a)$. Surprisingly the answer is `not always'. 
\begin{theorem}[Arithmetic Torelli theorem] \label{torelli}
There is a unique model $C/K$ of $\bar{C}$ such that:
\begin{enumerate}
\item
If $\bar{C}$ is hyperelliptic, there is an isomorphism
$$
\begin{CD}
(\Jac C, j) & @>{\sim}>> (A, a).
\end{CD}
$$
\item
If $\bar{C}$ is not hyperelliptic, there is a unique quadratic character
$\varepsilon$ of $\Gal(\bar{K}/K)$, and an isomorphism
$$
\begin{CD}
(\Jac C, j) & @>{\sim}>> (A, a)_{\varepsilon}
\end{CD}
$$
where $(A,a)_{\varepsilon}$ is the quadratic twist of $A$ by
$\varepsilon$. \end{enumerate}
\end{theorem}

\begin{remark}
It is tricky to find the right origin of the previous result. In \cite[Se.69]{serre-harvard}, Gouv\'ea indicates `Oort +\ldots' as a reference. One can indeed find in \cite[Lem.5.7]{oort89} a similar result (but this is 1991). Sekiguchi also worked on this question but after two errata, he gives in \cite{sekiguchi} only the existence of the model $C/K$ but does not speak about $\varepsilon$. One can also find this result in \cite[p.236]{mazur}.
\end{remark}

The notation $(A,a)_{\varepsilon}=(A_{\varepsilon},a_{\varepsilon})$ should be understood as follows. The variety $A_{\epsilon}$ is uniquely defined up to isomorphism by the following property: there exists a quadratic extension $L/K$ and an isomorphism $\phi : A \to A_{\epsilon}$ defined over $L$ such that for all $\sigma \in \Gal(\bar{K}/K)$ one has $\phi^{\sigma}=\varepsilon(\sigma) \phi$. The polarization $a_{\varepsilon}$ is the pull-back of $a$ by $\phi^{-1}$.\\

This result is a consequence of Weil's descent as explained in \cite[4.20]{serre-local} and of  Torelli theorem \cite[p.790-792]{matsusaka}. The schism which appears between the hyperelliptic and non hyperelliptic case is due to the fact that 
$$\Aut(\Jac C,j) \simeq \begin{cases} \Aut(C) & \textrm{if} \; C \; \textrm{is hyperelliptic,} \\
\Aut(C) \times \{\pm 1\} & \textrm{if} \; C \; \textrm{is non hyperelliptic.} 
\end{cases}
$$
\begin{definition}
The character $\varepsilon$ (or the discriminant of the extension $L/K$) is called \emph{Serre's obstruction}. By extension, in the hyperelliptic case, we say that $\varepsilon$ is trivial.
\end{definition}

Let us emphasize why this obstruction is an issue in our strategy. So far we have been able to prove in certain cases the existence of a geometrically indecomposable principally polarized abelian variety $(A,a)/k$ with Weil polynomial $(X^2+mX+q)^g$. Thanks to  Proposition \ref{ou}, we know that it is geometrically the Jacobian of a curve $\bar{C}$. If the obstruction is trivial, then $\bar{C}$ descends to a curve $C/k$ such that $\Jac C \simeq A$ and so $C$ is optimal. On the contrary, if the obstruction is not trivial, then $\bar{C}$ descends to a curve $C/k$ such that $\Jac C$ is isomorphic to the (unique) quadratic twist of $A$ and so its Weil polynomial is $(X^2-mX+q)^g$. In particular $\# C(k)=q+1-gm$ and $C$ is not optimal (and actually $C$ has  the minimum number of points a genus $g$ curve over $k$ can have). 

\subsection{The end of the genus $1$ and $2$ cases}
Since a genus $1$ curve over a finite field always has a rational point, it is an elliptic curve and Proposition \ref{prop-deu} tells us when an optimal genus $1$ curve exists (for the value of $N_q(1)$ see \cite{deuring} or \cite{serre-point}).\\  
When $g=2$, all genus $2$ curves are hyperelliptic so the obstruction is always trivial and the result is similar to Theorem \ref{theo-pp}, namely
\begin{theorem}[Serre]
There is no optimal curve of genus $2$ over $\FF_q$ if and only if $q=4$ or $9$ or $m^2-4q \in \{-3,-4,-7\}$.
\end{theorem}

In \cite{serre-point}, a closed formula for the value of $N_q(2)$ is given.
More recently, completing the work started by many  authors, we obtained in \cite{HNR09}  the complete picture for abelian surfaces, \ie we determined which isogeny classes contain the Jacobian of a genus $2$ curves in terms of the coefficients of the Weil polynomial.

\section{The genus $3$ case} \label{g3}

As there exist non hyperelliptic genus $3$ curves (the non-singular plane quartics), Serre's obstruction may not be trivial. One hope is that, for each $q$, there would be an optimal \emph{hyperelliptic} curve but this possibility has to be discarded: for instance, there does not exist any optimal hyperelliptic genus $3$ curves over $\FF_{2^{n}}$ with $n$ even since supersingular hyperelliptic curves do not exists in characteristic $2$ \cite{oort89}. Other counterexamples can be found in odd characteristic as well as it will be apparent in Proposition \ref{crazy}. Therefore, it is important to be able to compute Serre's obstruction. Currently, there is no perfect solution to this problem but we will summarize some of the ideas and partial answers which have been obtained. 

\subsection{Special families}
The key-idea is to use some families of curves with non trivial automorphisms, such that their Jacobian is a product of elliptic curves explicitly obtained as quotient by certain automorphism subgroups. Then one tries to reverse the process and see if one can glue given elliptic curves together to get a curve in the family. The possible quadratic extension one has to make during the construction  is Serre's obstruction. Let us illustrate this procedure with an example.

\begin{example}
The following family represents genus $3$ non hyperelliptic curves in characteristic $2$ with  automorphism group containing $(\ZZ/2\ZZ)^2$ 
$$C : (a (x^2+y^2) +c z^2 + xy +e z(x+y))^2=xyz(x+y+z), \quad ac(a+c+e) \ne 0.$$
The involutions are $(x:y: z)$ maps to  $(y:x:z), (x+z:y+z:z)$ or $(y+z:x+z:z)$. To get the equation of the curve $E_1=C/\langle (x:y:z) \mapsto (y:x:z)\rangle$, one introduces the invariant functions  $X=x+y,Y=xy$ and  finds
$$E_1: (a X^2 + c + Y+ e X)^2=Y(X+1).$$
Doing similarly with the other involutions and rewriting the  equations of the elliptic curves (see \cite{NR09} for details) one gets that  $\Jac C \sim E_1 \times E_2 \times E_3$ where
\begin{eqnarray*}
E_1 &: y^2+ xy &= x^3 + e x^2+ a^2 (a+c+e)^2, \\
E_2 &: y^2+ xy &= x^3 + e x^2+ c^2 (a+c+e)^2, \\
E_3 &: y^2+ xy &= x^3 + e x^2+ c^2 a^2.
\end{eqnarray*}
Conversely, we now want to glue ordinary elliptic curves $E_i$ with $j$-invariant $j_i \ne 0$. They can always be written $E_i : y^2+xy=x^3+ e x^2+ 1/j_i$ where, if $q>2$, $\Tr_{\FF_q/\FF_2}(e)=0$ if and only if $\Tr(E_i) \equiv 1 \pmod{4}$.
 Let $s_i^4=1/j_i$, then
\begin{eqnarray}
a &=& \frac{s_1 s_3}{s_2}, \nonumber \\
c &=& \frac{s_2 s_3}{s_1},  \nonumber \\
e &=& \frac{s_1 s_3}{s_2} + \frac{s_2 s_3}{s_1} + \frac{s_1 s_2}{s_3} \label{obsp2}.
\end{eqnarray}
Now, for instance, assume that $m \equiv -1 \pmod{8}$. This happens for  $n=35,37,63,\ldots$. We  choose $E=E_1=E_2=E_3$ an ordinary elliptic curve with trace $-m$ and $j$-invariant $j$ ($E$ exists since $2 \nmid m$). Since we can assume that $q>4$, the curve $E$ has an $8$-torsion point and it is not difficult to check that this implies (actually is equivalent to) $\Tr_{\FF_q/\FF_2} 1/j=0$. Hence
$$\Tr_{\FF_q/\FF_2}\left(\frac{s_1 s_3}{s_2} + \frac{s_2 s_3}{s_1} + \frac{s_1 s_2}{s_3}\right)=\Tr_{\FF_q/\FF_2}(1/j)=0.$$
On the other hand, since $\Tr(E) \equiv 1 \pmod{4}$, we have $\Tr_{\FF_q/\FF_2}(e)=0$ as well, so there is no obstruction to \eqref{obsp2}. Actually, we get an explicit equation
$$C : (j^{-1/4} (x^2+y^2+z^2+xz+yz)+xy)^2=xyz(x+y+z)$$
for the optimal curve.
\end{example}

Exploiting other families of curves in characteristic $2$, we get the following result.
\begin{theorem}[\cite{NR08,NR09}]
If $n$ is even, there exists an optimal curve over $\FF_{2^n}$ if and only if $n \geq 6$.\\
If $n$ is odd and $m \equiv 1,5,7 \pmod{8}$, there is an optimal curve over $\FF_{2^n}$.
\end{theorem}
When $n>1$ is odd and $m$ is even, there is of course no optimal curve since there is no elliptic curve with trace $-m$. So only the case $m \equiv 3 \pmod{8}$ is missing to get a complete answer when $p=2$.


More recently, Mestre \cite{mestreg3}  has worked with a family of curves with automorphism group $S_3$ and showed that if $p=3$ (resp. $ p=7$), $3 \nmid m$ (resp. $3 | m$) and $-m$ is a non-zero square modulo $7$ (resp. $n \geq 7$), then there exists an optimal curve over $\FF_{3^n}$ (resp. $\FF_{7^n}$).

To conclude on this approach, let us point out  that one could use the family with automorphism group $(\ZZ/2\ZZ)^2$ (called \emph{Ciani quartics}) also in characteristic greater than $2$, since Serre's obstruction has been worked out in \cite{HLP}. Unfortunately, one does not see when this obstruction is trivial knowing only $m$ (one needs the equations of the elliptic factors to decide).

\subsection{Serre's analytic strategy}
Inspired by results of Klein \cite[Eq.118,p.462]{klein} and Igusa \cite[Lem.10,11]{igusag3}, in a 2003 letter to Jaap Top \cite{LR}, Serre stated a strategy to compute the obstruction when the characteristic is different from $2$. Roughly speaking, his idea was that  a certain Siegel modular form evaluated at a `moduli point' $(A,a)/K$ is a square in $K$ if and only if the obstruction is trivial. In a series of three papers, it was shown that this is accurate (first for  Ciani quartics, then in general) and how to compute the obstruction in the case of the power of a CM elliptic curve. Let us  state the general  result without any comments on the proof which would lead us to far from our initial purpose (see however \cite[Chap.4]{HDRritz} for details).

\begin{theorem}[\cite{LRZ}] \label{LRZth}
Let $A = (A,a)/K$ be a principally polarized abelian threefold defined over a field $K$ with $\chr K \ne 2$. Assume that $a$ is geometrically indecomposable. There exists a unique primitive geometric Siegel modular form of weight $18$ defined over $\ZZ$,  denoted $\chi_{18}$, such that 
\begin{enumerate}[i)]
 \item $(A,a)$ is a hyperelliptic Jacobian if and only if $\chi_{18}(A,a)=0$.
\item $(A,a)$ is a non hyperelliptic Jacobian if and only if $\chi_{18}(A,a)$ is a non-zero square. \label{LRZcases2}
\end{enumerate}
Moreover, if $K \subset \CC$,  let 
\begin{itemize}
\item $(\omega_{1}, \omega_2, \omega_{3})$ be a basis of regular differentials on $A$;
\item $\gamma_{1}, \dots \gamma_{6}$ be a symplectic basis (for $a$) of $H_1(A,\ZZ)$;
\item $\Omega_a := [\Omega_{1} \ \Omega_{2}] =[\int_{\gamma_j} \omega_i]$ be a period matrix with $\tau_a : = \Omega_{2}^{- 1} \Omega_{1} \in \HH_{3}$ a Riemann matrix.
\end{itemize}
Then $(A,a)$ is a Jacobian if and only if
\begin{equation} \label{chiex}
\chi_{18}((A,a),\omega_1\wedge \omega_2 \wedge \omega_3):= \frac{(2\pi)^{54}}{2^{28}} \cdot \frac{\prod_{[\epsm]} \theta[\epsm](\tau_a)}{\det(\Omega_2)^{18}}
\end{equation}
is a square in $K$. 
\end{theorem}
Let us recall that the \emph{Thetanullwerte} $\theta[\epsm](\tau)$ are the $36$ constants  such that
$$
[\epsm] = \Ch{\epsilon_{1}}{\epsilon_{2}} \in \{0,1\}^3 \oplus \{0, 1\}^3,
$$
with   $\epsilon_1  {^t \epsilon_2} \equiv 0 \pmod{2}$ 
and for
$\tau \in \HH_{3}$
$$\theta  \Ch{\epsilon_{1}}{\epsilon_{2}}(\tau) =
\sum_{n \in \ZZ^3}
\exp(i \pi (n + \epsilon_{1}/2)  \tau {^t (n + \epsilon_{1}/2)} +  i \pi
(n + \epsilon_{1}/2){^t \epsilon_{2}}).
$$

\begin{remark}
For a different approach on this result, see \cite{meagher}.
\end{remark}

The initial aim of Serre's letter was of course the existence of optimal curves of genus $3$. 
However, one does not know how to compute directly the value of $\chi_{18}$ over finite fields. Therefore, as Serre suggested, when $A$ is ordinary, we lift $(A,a)$ canonically over a number field and there, we  use formula \eqref{chiex}. Doing the computation with enough precision, we can recognize this value as an algebraic number. Finally we reduce it to the initial finite field to see if it is a square.

As the Jacobian of an optimal curve is isogenous to the power of an elliptic curve $E$, in \cite{Rit-serre}, we worked out this procedure explicit in the particular case  $A=E^3$. Let $a_0$ be the product principal polarization on $E^3$ and $M=a_0^{-1} a \in M_3(\End(E))$. When $\End(E)$ is an order in an imaginary quadratic field, it is well known that $M$ is the matrix of a principal polarization on $E^3$ if and only if $M$ is a positive definite hermitian matrix with determinant $1$ (see Remark \ref{othereq} and \cite[p.209]{mumford}). Moreover,  when $E$ is defined over a number field,  we  show how to translate the data $(E^3,a_0 M)$ into a  period matrix of the corresponding torus in order to compute the analytic expression of $\chi_{18}$. Let us illustrate this procedure with the following example.

\begin{example} \label{ex19}
Does there exist an optimal curve $C$ of genus $3$ over $k=\FF_{47}$ ? If so, by Lemma \ref{toE3} we know that $\Jac C$ is isogenous to $E^3$ where $E$ is an elliptic curve with trace $-\lfloor 2 \sqrt{47} \rfloor=-13$. The curve $E$ is then an ordinary elliptic curve and $\End(E)$ contains $\ZZ[\pi] \simeq \ZZ[(13+\sqrt{13^2-4 \cdot 47})/2]=\ZZ[\tau]$ (where $\pi$ is the $k$-Frobenius endomorphism and $\tau=(1+\sqrt{-19})/2$). Hence $\End(E)=\ZZ[\pi]$ is the ring of integers $\oo_L$ of $L=\QQ(\sqrt{-19})$. Since $\oo_L$ is principal, $E$ is unique up to isomorphism.
Using the work of \cite{schiemann}, one can see that, up to automorphism, there is a unique positive definite hermitian matrix $M \in M_3(\oo_L)$ of determinant $1$ which is indecomposable. In the language of Section \ref{equivalence}, this means that there exists a unique positive definite unimodular indecomposable rank $3$ hermitian $\oo_L$-module. The abelian threefold $(E^3,a_0 M)$ is then the unique principally polarized  geometrically indecomposable abelian threefold with Weil polynomial $(X^2+13 X+47)^3$, up to isomorphism.\\
Lifting $E$ canonically over $\QQ$ as 
$\bar{E} : y^2=x^3-152 x-722$
we can consider the principally abelian threefold $(\bar{E}^3,a_0 M)$ since $\End(\bar{E})=\oo_L$ as well. Let $[w_1 \; w_2]$ be a period matrix of $E$ with respect to the canonical regular differential $dx/(2y)$. If we let 
$$\Omega_0= \left[ \
\left(\begin{array}{ccc}
w_{1} & 0           & 0           \\
0           & w_{1} & 0           \\
0           & 0           & w_{1}
\end{array}\right)
\left(\begin{array}{ccc}
w_{2} & 0           & 0          \\
0           & w_{2} & 0          \\
0           & 0           & w_{2}
\end{array}\right)
\right],$$
$\CC^3/\Omega_0 \ZZ^6 \simeq \bar{E}^3(\CC)$ with the product polarization $a_0$.  We then need to find a symplectic basis of $\Omega_0 \ZZ^6$ for the polarization $a_0 M$. 
It is  not difficult to prove that the first Chern class of $a_0 M$ with respect to the pull-back $\omega_i$ of the differentials $dx/(2y)$ on each curve is represented by the matrix 
$$H=\frac{1}{w_1 \overline{w}_2} {^t M}.$$
The alternated form $T$ classically associated to $H$ on the lattice  $\Omega_0 \ZZ^{6}$ is 
$T=\im(\tp{\Omega_0} H \overline{\Omega}_0)$.
One then finds a matrix $B \in \GL_{6}(\ZZ)$ such that 
$$BT\tp{B}=\left[\begin{array}{cc} 0 & \bfI_3  \\ -\bfI_3 & 0 \end{array}\right]$$ and $\Omega=\Omega_0 \tp{B}$ is a period matrix for the polarization $a_0 M$. Finally, one computes an approximation of   $$\chi=\chi_{18}((\bar{E}^3,a_0 M),\omega_1 \wedge \omega_2 \wedge \omega_3)$$ thanks to the analytic formula \eqref{chiex} and we recognize it as an element of $L$. We find in our case
$$\chi=(2^{19} \cdot 19^7)^2.$$
The value $\chi$ is a non-zero square over $\FF_{47}$ so by Theorem \ref{LRZth} \eqref{LRZcases2} Serre's obstruction is trivial and there is a non hyperelliptic optimal curve of genus $3$ over $k$.
\end{example}

Similar computations show that there is an optimal curve over $\FF_q$ for $q=61,137,277$ but not for $q=311$. Note that this result for $q=47$ and $q=61$ has already been obtained in \cite{top} using explicit models and the others have been confirmed by \cite{alek}. In \cite{Rit-serre},  tables of values of $\chi$ as the one from Example \ref{ex19} are given for $(\bar{E}^3,a_0 M)$ where $\bar{E}$ is an elliptic curve with class number $1$ and $M$ is taken from \cite{schiemann}. From them, we can get for instance:

\begin{proposition} \label{crazy}
Assume that $q=p^n$ is  such that $4q=m^2+d$ with $d=7$ (resp. $19$). Then there exists an (explicit) genus $3$ optimal curve over $\FF_q$ if and only if 
$$m \equiv 1,2 \;  \textrm{or} \; 4 \pmod{7} \; \textrm{(resp.} \;  \left(\frac{m}{19}\right) \left(\frac{-2}{p}\right)=1 \textrm{)}.$$
Moreover if this curve exists, it is non hyperelliptic.\\
Assume that $q=p^n$ is such that $4q=m^2+43$. In particular $43$ is a square in $\FF_p$, let say $43=r^2$ with $r \in \FF_p$. Then there exists a  genus $3$ optimal curve over $\FF_q$ if and only if 
$$ \left(\frac{m}{43}\right) \left(\frac{\alpha}{p}\right)=1$$
where $\alpha$ is either $-2\cdot 3 \cdot 7,-487,-47 \cdot 79 \cdot 107 \cdot 173$ or $-15156 \pm 8214 r$. Moreover if this curve exists, it is non hyperelliptic.  
\end{proposition}

\begin{remark}
The term `explicit' in Proposition \ref{crazy} comes from the fact that for certain $(\bar{E}^3,a_0 M)$ of \cite{Rit-serre} we were able to give the equation of a curve $\bar{C}$ such that $\Jac \bar{C}$ is isomorphic to $(\bar{E}^3,a_0 M)$  using \cite{guardia}. Hence for these cases, we have a `universal' family of explicit equations for the optimal curve.
\end{remark}

The fact that the previous statement is embarrassedly cumbersome reveals either the intrinsic difficulty of the problem or a wrong attack angle. Moreover, 
the limits of this strategy already  appear in the example: the computation of the canonical lift, of the matrices $M$ and of a period matrix make it algorithmic in nature. Worse, the computation of an approximation of $\chi$ is time-consuming since one has to recognize it as an algebraic number (actually for a good choice of the model $\bar{E}$, $\chi$ is an algebraic integer). Therefore large values of the discriminant of $\End(E)$ seem out of reach.\\
 It might then be interesting to try to understand the prime decomposition of $\chi$ algebraically. Klein's formula linking $\chi_{18}$ to the square of the discriminant of plane quartics (see  \cite[Eq.118,p.462]{klein} and \cite[Th.2.23]{LRZ}) makes us think about an analogue of the N\'eron-Ogg-Shafarevich formula for elliptic curve \cite[Appendix C,16]{silverman}. We shall then interpret $\mathfrak{p} | \chi$ in terms of the nature of $(A,a):=(\bar{E}^3, a_0 M) \pmod{\mathfrak{p}}$. For instance, using \cite[p.1059]{ichi2}, $\mathfrak{p} | \chi$ if and only if $(A,a)$ is geometrically decomposable or a hyperelliptic Jacobian. Unfortunately, we do not know how to detect algebraically this last possibility (see the discussion in \cite[Sec.4.5.1]{HDRritz}).
\begin{remark}
We have not spoken yet about the case $q$ square when $p>2$. First, when $p \equiv 3 \pmod{4}$, one knows \cite[p.2]{ibug3} that there exists an optimal genus $3$ curve. This curve is even hyperelliptic 
\cite{oort89} but not explicit (see however \cite{kodama} for some explicit sub-cases). Also  Fermat curve $x^4+y^4+z^4=0$ is optimal if $n \equiv 2 \pmod{4}$. Then, when $p \equiv 1 \pmod{4}$ and $n \equiv 2 \pmod{4}$, Ibukiyama \textit{(loc. cit.)} shows that there is  an optimal curve. Ibukiyama's strategy uses a mass formula on quaternion hermitian forms to show the descent of an indecomposable  principal polarization on a model over $\FF_p$ of  $E^3$  where $E/\FF_{p^2}$ is an elliptic curve with trace $-2p=-m$. The abelian threefold and its quadratic twist being isomorphic over $\FF_{p^2}$, he avoids the issue of computing Serre's obstruction.
\end{remark}

\subsection{The geometric approach}
Following a construction of Recillas \cite{recillasprym}, we were able  to give in \cite{BeauRit} a geometric characterization of Serre's obstruction. For the sake of simplicity, let us assume that $\chr k \ne 2$ and that $(A,a)/k$ is geometrically the Jacobian of a non hyperelliptic genus $3$ curve. Since $k$ is a finite field, there exists a symmetric theta divisor $\Theta$ (for the polarization $a$) defined over $k$. Let $\Sigma$ be the union of $2_* \Theta$ and of the unique divisor in $|2\Theta|$ with multiplicity greater than or equal to $4$ at $0$. 
\begin{proposition}
Let $\alpha   \in A(\bar{k}) \setminus \{0\}$. The curve $\tilde{X}_{\alpha}=\Theta \cap (\Theta+\alpha)$ is smooth and connected if and only if $\alpha \in A(\bar{k}) \setminus \Sigma$.
\end{proposition}
Hence, the divisor $\Sigma$ represents a bad locus that needs to be avoided in the sequel. Assuming that $\alpha \notin \Sigma$, the involution $(z \mapsto \alpha-z)$ of $\tilde{X}_{\alpha}$ is fixed point free and so  $X_{\alpha}=\tilde{X}_{\alpha}/(z \mapsto \alpha-z)$ is a smooth genus $4$ curve.
\begin{proposition}
The curve $X_{\alpha}$ is non hyperelliptic and its canonical model in $\PP^3$ lies on a quadric $Q_{\alpha}$ which is smooth.
\end{proposition}
To go further, we need to assume that $\alpha$ is rational. When $k$ is big enough, such an $\alpha$ always exists. We then obtain the following result.
 \begin{theorem}
Assume there exists $\alpha \in A(k) \setminus \Sigma$. Then $(A,a)$ is a Jacobian if and only if $\delta=\Disc Q_{\alpha}$ is a square in $k^*$.
 \end{theorem}
Let us sketch the proof. A non hyperelliptic genus $4$ curve $X$ lies canonically in $\PP^3$ on the intersection of a unique quadric $Q$ and a cubic surface $E$. If we assume that $Q$ is smooth, then $X$ has two  $g^1_3$ coming from the two rulings of $Q$ by intersecting them with $E$. Moreover, an easy computation shows that $\Disc Q$  is a square if and only if these two $g^1_3$ are defined over $k$. Now Recillas' construction, which can be used when
$(A,a)$ is the Jacobian of a curve, shows that $X_{\alpha}$ has two (rather explicit) rational $g^1_3$. To conclude, it is then enough to show that a quadratic twist of $(A,a)$ (which is no more a Jacobian) leads to two conjugate $g^1_3$.\\

The advantage of this approach is that it stays over the finite field $k$ and is completely algebraic. Unfortunately, so far, we do not see how to compute $\delta$ for  $A=E^3$ and $a=a_0 M$ in terms of an equation of $E$ and the coefficients of $M$. The main difficulty seems to find an equation (or even points) on a theta divisor in order to compute an equation of $Q_{\alpha}$.

\bibliographystyle{alpha}
\bibliography{synthbib}

\end{document}